\documentclass{amsart}
\input{setup/preamble.sty}

\begin{document}
    \title{A saturated 1-system of curves on the surface of genus 3}
    \author{Zhaoshen Zhai}
    \address{Department of Mathematics and Statistics, McGill University, 805 Sherbrooke Street West, Montreal, QC, H3A 0B9, Canada}
    \email{zhaoshen.zhai@mail.mcgill.ca}
    \thanks{This work was partially supported by the Rubin Gruber Science Undergraduate Research Award grant for summer 2023.}
    \date{\today}

    \begin{abstract}
        We construct a system of $33$ essential simple closed curves that are pairwise non-homotopic and intersect at most once on the oriented, closed surface of genus $3$. Moreover, we show that our construction is \textit{saturated}, in the sense that it is not properly contained in any other such system of curves.
    \end{abstract}

    \maketitle

    Let $\Sigma_g$ be the oriented, closed surface of genus $g\geq1$, and let $k\geq0$. Following Juvan, Malni\v{c}, and Mohar \cite{JMM96}, we define a \textit{$k$-system of curves} on $\Sigma_g$ to be a collection of essential simple closed curves that are pairwise non-homotopic and intersect at most $k$ times; see Section \ref{sec:preliminaries} for preliminaries. Let $N(k,g)$ denote the maximum cardinality of a $k$-system on $\Sigma_g$, which was shown to be finite in the same paper using techniques from topological graph theory. They considered the following question.

    \begin{question*}
        How large can a $k$-system on $\Sigma_g$ be? In other words, what is $N(k,g)$?
    \end{question*}

    In particular, Farb and Leininger asked for the asymptotics of this question for $k=1$, and in this case, Przytycki \cite{Prz15} established a cubical bound $N(1,g)=O(g^3)$. This bound was later improved by Aougab, Biringer, and Gaster \cite{ABG17} to $O(g^3/(\log g)^2)$, and finally by Greene \cite{Gre19}, who established the current best known upper bound of $O(g^2\log g)$.

    Other than the torus $\Sigma_1$, which is easily seen to have $N(1,1)=3$, Malestein, Rivin, and Theran \cite{MRT14} explicitly determined $N(1,2)=12$ using the hyperelliptic involution on $\Sigma_2$. For $g=3$, \cite{MRT14} also proved that $18=6g\leq N(1,3)\leq(g-1)(2^{2g}-1)=126$, but this is far from optimal.\\

    In this paper, we improve the above result on $N(1,3)$ in two directions. First, we show that $N(1,3)\geq33$ via an explicit construction of a $1$-system $X$ with $33$ curves. Second, we show that $X$ is \textit{saturated}, in that it is not properly contained in any other $1$-system on $\Sigma_3$. However, we do not know whether $N(1,3)=33$, but we do suggest a possible route towards a positive result in Section \ref{sec:maximality}.

    \subsection*{Acknowledgements}

    I thank my advisor, Piotr Przytycki, for his consistent support throughout the summer of 2023, and for teaching me (the very basics of) geometric topology. This $1$-system of $33$ curves was first found by Chris Leininger and Piotr Przytycki in 2011 (personal communication). I also thank the referee for all the remarks and corrections that have greatly improved the article; in particular, I thank the referee for the short and elegant proof of the saturation claim in Proposition \ref{prp:type_triangle_24}.

    \section{Preliminaries}\label{sec:preliminaries}

    We give a brief survey of the relevant definitions and results.

    \subsection{Curves and surfaces}

    For the purposes of this paper, a \textit{surface} is a connected, oriented, $2$-dimensional manifold of finite type, possibly with boundary, but without punctures. These surfaces are all compact, and up to homeomorphism, there is a unique surface $\Sigma_{g,b}$ with genus $g$ and with $b$ boundary components. It is a fact that every surface $\Sigma_{g,b}$ with negative Euler characteristic admits a hyperbolic metric with geodesic boundary, so we call $\Sigma_{g,b}$ \textit{hyperbolizable}. Throughout this section, let $\Sigma$ be a hyperbolizable surface and let $\del\Sigma$ denote its (possibly empty) boundary.

    A \textit{(closed) curve} on $\Sigma$ is a continuous map $S^1\to\Sigma$, which is \textit{simple} if it embeds in $\Sigma$ $-$ in which case we identify it with its image $-$ and is \textit{essential} if it is not homotopic to a point. Unless otherwise stated, all curves are assumed to be simple and essential.

    \subsection{Dehn twists}

    Following \cite{Wil21}, consider an annular neighborhood $\overline{N}_\alpha\iso S^1\times\l[0,1\r]\subset\Sigma$ of a curve $\alpha$ on $\Sigma$, with $\alpha$ being the `core curve' $S^1\times\l\{1/2\r\}$. The \textit{cut surface} of $\Sigma$ along $\alpha$ is the surface $\Sigma_\alpha\coloneqq\Sigma\comp N_\alpha$ obtained by removing the interior of the annulus, which induces two boundary curves $\alpha^+$ and $\alpha^-$ on $\del\Sigma_\alpha$. The surface $\Sigma$ is then reconstructed by glueing the curves $\alpha^+$ and $\alpha^-$ by an annulus $A\coloneqq S^1\times\l[0,1\r]$. The curve $\alpha$ is said to be \textit{non-separating} if $\Sigma_\alpha$ is connected, and is \textit{separating} otherwise.

    Fix an orientation-preserving homeomorphism $\phi:A\to\overline{N}_\alpha$, where the orientation on $A$ is induced by the standard orientation of the plane. Following \cite{FM12}, we consider the \textit{twist} map $T:A\to A$ given by $T(\theta,t)\coloneqq\tpl{\theta+2\pi t,t}$. This is an orientation-preserving homeomorphism fixing $\del A$ pointwise, and lifts to a map $T_\alpha:\Sigma\to\Sigma$ given by $T_\alpha(x)\coloneqq(\phi\circ T\circ\phi^{-1})(x)$ for every $x\in N_\alpha$, which fixes every point outside of $N_\alpha$. The mapping class of $T_\alpha$ depends only on the homotopy class $a$ of $\alpha$, and is called the \textit{Dehn twist in $a$}. Our choice of the twist map $T$ gives us $T_\alpha$ as the \textit{left} Dehn twist; the \textit{right} Dehn twist is its inverse.

    \subsection{Arcs}

    An \textit{arc} on $\Sigma$ is a continuous map $\l[0,1\r]\to\Sigma$. We will mostly be interested in arcs whose endpoints lie on boundary components, in which case we say that the arc is \textit{between} these boundary components. An arc is \textit{simple} if it embeds in $\Sigma$, in which case we identify it with its image. Two arcs $\alpha$ and $\beta$ on $\Sigma$ that start at the same boundary component and end at the same boundary component are said to be \textit{homotopic} if there is a free homotopy $H:\l[0,1\r]\times\l[0,1\r]\to\Sigma$ between $\alpha$ and $\beta$ that slides the endpoints along the respective boundary components, in which case we write $\alpha\htopeq\beta$.

    \subsection{The Subsurface Projection}

    Let $\mc{C}_0(\Sigma)$ denote the collection of all (homotopy classes of) curves on $\Sigma$; this is the $0$-skeleton of the \textit{complex of curves} $\mc{C}(\Sigma)$. For a curve $\alpha$ on $\Sigma$, let $\Sigma'$ be a connected component of the cut surface of $\Sigma$ along $\alpha$ and let $\mc{A}_0(\Sigma')$ be the collection of all (homotopy classes of) simple arcs on $\Sigma'$. Suppose that $\Sigma=\Sigma_{g,b}$ is such that $3g+b>3$. Following \cite{MM00}*{Section 2.3}, we consider the \textit{subsurface projection} $\pi:\mc{C}_0(\Sigma)\to\ms{P}(\mc{A}_0(\Sigma'))$ to the power set of $\mc{A}_0(\Sigma')$, defined by sending a curve $\gamma\in\mc{C}_0(\Sigma)$ to the union of (homotopy classes of) simple arcs obtained by intersecting $\gamma$ with $\Sigma'$. Note that $\pi(\gamma)=\varnothing$ if $\gamma$ can be homotoped away from $\Sigma'$. In the situation concerned in this paper $-$ where $\pi(\gamma)$ is always a singleton $-$ we abuse notation and identify $\pi(\gamma)$ with its image.

    \subsection{Intersection numbers}

    Two curves $\alpha$ and $\beta$ on $\Sigma$ are said to be in \textit{minimal position} if $\l|\alpha\cap\beta\r|$ cannot be decreased by a homotopy. For the respective (free) homotopy classes $a$ and $b$ of $\alpha$ and $\beta$, their \textit{geometric intersection number} is defined to be $i(a,b)\coloneqq\min\l\{\l|\alpha\cap\beta\r|\st\alpha\in a\textrm{ and }\beta\in b\r\}$. In other words, $i(a,b)$ is the intersection number of representatives of $a$ and $b$ in minimal position. We note that $i(a,a)=0$, since any $\alpha,\alpha'\in a$ can be homotoped to be disjoint. If $\alpha$ is separating, then $i(a,b)$ is even for any $b$. Abusing notation, we write $i(\alpha,\beta)$ for $i(a,b)$, where $\alpha\in a$ and $\beta\in b$.

    The geometric intersection number can similarly be defined for (homotopy classes of) arcs. Recall that we allow homotopies of arcs between boundary components to slide the endpoints along the respective boundary components.

    \subsection{The Bigon Criterion}

    Following \cite{Wil21}, two curves $\alpha$ and $\beta$ on $\Sigma$ are said to bound a \textit{bigon} if there is an embedded disk in $\Sigma$ whose interior is disjoint from $\alpha\cup\beta$ and whose boundary is the union of two arcs from $\alpha$ and $\beta$ intersecting at two points. The well-known \textit{Bigon Criterion} \cite{FM12}*{Proposition 1.7} states that two curves $\alpha$ and $\beta$ are in minimal position if and only if they do not bound a bigon.

    It is a fact from Riemannian geometry that every free homotopy class of curves on a compact Riemannian manifold admits a geodesic representative. For hyperbolic surfaces, this geodesic representative is unique, and geodesic representatives of simple curves are again simple \cite{FM12}*{Propositions 1.3, 1.6}. Thus, by the Bigon Criterion, any two distinct simple geodesics on a hyperbolic surface are in minimal position.

    There is a Bigon Criterion for arcs too \cite{FM12}*{Section 1.2.7}, but now we also need to consider \textit{half-bigons}: two arcs $\alpha$ and $\beta$ on $\Sigma$ starting at the same boundary component $\gamma$ are said to bound a \textit{half-bigon} if there is an embedded triangle in $\Sigma$ whose boundary is the union of two arcs from $\alpha$ and $\beta$ intersecting at a vertex, and whose opposite side is an arc of $\gamma$. The Bigon Criterion then states that two arcs are in minimal position if and only if they do not bound any bigons and half-bigons.

    \subsection{Systems of curves and arcs}

    For a fixed $k\geq0$, a \textit{$k$-system} on $\Sigma$ is a collection $X\coloneqq\l\{\alpha_j\r\}$ of pairwise non-homotopic essential simple curves such that $i(\alpha_j,\alpha_l)\leq k$ for all $j\neq l$. We say that $X$ is \textit{maximal} if $\l|Y\r|\leq\l|X\r|$ for any $k$-system $Y$, and is \textit{saturated} if $X\cup\l\{\gamma\r\}$ is not a $k$-system for any curve $\gamma\not\in X$. If $\Sigma=\Sigma_{g,b}$, we let $N(k,g,b)$ be the cardinality of a maximal $k$-system on $\Sigma$; if $b=0$, write $N(k,g)\coloneqq N(k,g,0)$.

    We similarly define a \textit{$k$-system of arcs} as a collection of simple arcs that are pairwise non-homotopic and intersect at most $k$ times. Given a $k$-system of curves $X$ and a subsurface $\Sigma'\subseteq\Sigma$, we let $\pi(X)$ denote the corresponding $k$-system of projected arcs on $\Sigma'$.

    \section{Construction of a Saturated 1-System}\label{sec:construction}

    In this section, we construct a saturated $1$-system on $\Sigma_3$ containing $33$ curves. To do so, we first describe a correspondence between the trivalent multigraphs on $4$ vertices and the pants decompositions of $\Sigma_3$, which will be used throughout the paper. Unless otherwise stated, we let $\Sigma\coloneqq\Sigma_3$ be the surface of genus $3$.

    \begin{definition}
        Let $\mc{P}$ be a pants decomposition of a surface. The \textit{dual graph} of $\mc{P}$ is the multigraph whose vertices are the pairs of pants in $\mc{P}$, where two vertices are joined by an edge for each pants curve in $\mc{P}$.
    \end{definition}

    Observe that the dual graph of any pants decomposition of $\Sigma$ is trivalent since each pair of pants is glued along all three of its boundary curves, and we leave it to the reader to check that the association of the pants decompositions of $\Sigma$ to the trivalent multigraphs on $4$ vertices is a bijection.\\

    Throughout this section, we use the pants decomposition $\mc{P}_0$ of $\Sigma$ whose dual graph is $K_4$, which determines $6$ pairwise disjoint curves $\alpha_1,\dots,\alpha_6$ on $\Sigma$ as the pants curves in $\mc{P}_0$ (see Figure \ref{fig:tetrahedral_correspondence}). Note that every $n$-cycle $C\subset K_4$ is the dual graph of the pants decomposition of the $n$-holed torus $\Sigma_{1,n}$ obtained by glueing the $n$ pairs of pants of $\mc{P}_0$ via the edges of $C$ (see Figure \ref{fig:cycle_holed_torus}).

    \begin{figure}
        \begin{tikzpicture}
            \begin{scope}[xshift=0.3in]
                \shell
                \def\edgeA{black}\def\edgeB{black}\def\edgeC{black}\def\edgeD{black}\def\edgeE{black}\def\edgeF{black}
                \def\drawEdgeA{red}\def\drawEdgeB{red}\def\drawEdgeC{red}\def\drawEdgeD{red}\def\drawEdgeE{red}\def\drawEdgeF{red}
                \edges
            \end{scope}
            \draw[<->] (3.4,0.3) -- (4.2,0.3);
            \begin{scope}[xshift=2in, yshift=0.55in, scale=1.4]
                \draw[black] (1, 0.000) -- (0,-1.730);
                \fill[white] (0.5,-0.865) circle (0.125cm);
                \draw[black] (0.5,-0.865) circle (0cm) node{\footnotesize{$\alpha_1$}};

                \draw (0,-1.73) -- (2,-1.73);
                \fill[white] (1,-1.73) circle (0.175cm);
                \draw (1,-1.73) circle (0cm) node{\footnotesize{$\alpha_2$}};

                \draw[black] (2,-1.730) -- (1, 0.000);
                \fill[white] (1.5,-0.865) circle (0.125cm);
                \draw[black] (1.5,-0.865) circle (0cm) node{\footnotesize{$\alpha_3$}};

                \draw (1,-1.173) -- (1, 0.000);
                \fill[white] (1,-0.5865) circle (0.12cm);
                \draw (1,-0.5865) circle (0cm) node{\footnotesize{$\alpha_4$}};

                \draw[black] (1,-1.173) -- (0,-1.730);
                \fill[white] (0.5,-1.4515) circle (0.15cm);
                \draw[black] (0.5,-1.4515) circle (0cm) node{\footnotesize{$\alpha_5$}};

                \draw[black] (1,-1.173) -- (2,-1.730);
                \fill[white] (1.5,-1.4515) circle (0.15cm);
                \draw[black] (1.5,-1.4515) circle (0cm) node{\footnotesize{$\alpha_6$}};

                \fill[black] (0,-1.73) circle (0.04cm);
                \fill[black] (2,-1.73) circle (0.04cm);
                \fill[black] (1,0) circle (0.04cm);
                \fill[black] (1,-1.173) circle (0.04cm);
            \end{scope}
        \end{tikzpicture}
        \caption{The pants decomposition $\mc{P}_0$ of $\Sigma$, whose pants curves $\alpha_1,\dots,\alpha_6$ are marked in red, and its dual graph $K_4$.}
        \label{fig:tetrahedral_correspondence}
    \end{figure}

    \begin{figure}
        \begin{tikzpicture}
            \begin{scope}[xshift=0.3in]
                \shell

                \fill[white] (-0.3,-0.5) rectangle (0.3,0.545);
                \begin{scope}[rotate=120]
                    \fill[white] (-0.3,-0.5) rectangle (0.3,0.545);
                \end{scope}
                \begin{scope}[rotate=240]
                    \fill[white] (-0.3,-0.5) rectangle (0.3,0.545);
                \end{scope}

                \def\edgeA{black}\def\edgeB{black}\def\edgeC{black}
                \def\drawEdgeA{red}\def\drawEdgeB{red}\def\drawEdgeC{red}\def\drawEdgeD{black}\def\drawEdgeE{black}\def\drawEdgeF{black}
                \edges
            \end{scope}
            \draw[<->] (3.4,0.3) -- (4.2,0.3);
            \begin{scope}[xshift=2in, yshift=0.55in, scale=1.4]
                \draw[black] (1, 0.000) -- (0,-1.730);
                \fill[white] (0.5,-0.865) circle (0.125cm);
                \draw[black] (0.5,-0.865) circle (0cm) node{\footnotesize{$\alpha_1$}};

                \draw (0,-1.73) -- (2,-1.73);
                \fill[white] (1,-1.73) circle (0.175cm);
                \draw (1,-1.73) circle (0cm) node{\footnotesize{$\alpha_2$}};

                \draw[black] (2,-1.730) -- (1, 0.000);
                \fill[white] (1.5,-0.865) circle (0.125cm);
                \draw[black] (1.5,-0.865) circle (0cm) node{\footnotesize{$\alpha_3$}};

                \fill[black] (0,-1.73) circle (0.04cm);
                \fill[black] (2,-1.73) circle (0.04cm);
                \fill[black] (1,0) circle (0.04cm);
            \end{scope}
        \end{tikzpicture}
        \caption{The induced pants decomposition of $\Sigma_{1,3}\subset\Sigma$, whose dual graph is a $3$-cycle in $K_4$. For a $4$-cycle in $K_4$, its corresponding pants decomposition of $\Sigma_{1,4}\subset\Sigma$ can be obtained by cutting $\Sigma$ along the curves corresponding to the two edges of $K_4$ not in the $4$-cycle.}
        \label{fig:cycle_holed_torus}
    \end{figure}

    \subsection{Encoding Dehn twists}\label{sec:encoding_dehn_twists}

    Using the pants decomposition $\mc{P}_0$ of $\Sigma$, we can classify curves on $\Sigma$ by the number of distinct pairs of pants in $\mc{P}_0$ that it intersects. The trivial case is if a curve intersects exactly one pair of pants $P\in\mc{P}_0$, in which case it is homotopic to a boundary curve of $P$.

    Let $X_0$ be any $1$-system of curves on $\Sigma$ \textit{supported} by $\mc{P}_0$, that is, we require that $X_0$ contains all six pants curves $\alpha_1,\dots,\alpha_6$ of $\mc{P}_0$. If $\gamma\in X_0$, then it cannot intersect exactly two pairs of pants in $\mc{P}_0$. Indeed, if the two pairs of pants that $\gamma$ intersects are glued along a curve $\alpha$ to form a sphere $\Sigma_{0,4}$ with four holes, then $\alpha$ is a separating curve of $\Sigma_{0,4}$. Thus $i(\gamma,\alpha)$ is even, and since it is non-zero, we see that $i(\gamma,\alpha)>1$.

    Every other curve $\gamma\in X_0$ intersects either three or four pairs of pants in $\mc{P}_0$. In either case, we consider the embedding $K_4\into\Sigma$ as shown in Figure \ref{fig:embedding_graph_in_surface}, which we use to pin down `reference curves' $\gamma_0$ on $\Sigma$ so that $\gamma$ is a multitwist thereof. Note that we can also view $\Sigma$ as the boundary of a regular neighborhood of $K_4$, so that the natural action $A_4\act K_4$ extends to an action $A_4\act\Sigma$ by orientation-preserving isometries.

    \begin{figure}
        \begin{tikzpicture}
            \shell

            \begin{scope}
                \draw[blue] (90:1+\width) .. controls (90:1+\width-0.1) .. (150:0.89) .. controls (210:1+\width-0.1) .. (210:1+\width);
                \draw[blue, dash pattern={on 2pt off 1pt}, dash phase=1pt] (90:1+\width) -- (0,0);
            \end{scope}
            \begin{scope}[rotate=120]
                \draw[blue] (90:1+\width) .. controls (90:1+\width-0.1) .. (150:0.89) .. controls (210:1+\width-0.1) .. (210:1+\width);
                \draw[blue, dash pattern={on 2pt off 1pt}, dash phase=1pt] (90:1+\width) -- (0,0);
            \end{scope}
            \begin{scope}[rotate=240]
                \draw[blue] (90:1+\width) .. controls (90:1+\width-0.1) .. (150:0.89) .. controls (210:1+\width-0.1) .. (210:1+\width);
                \draw[blue, dash pattern={on 2pt off 1pt}, dash phase=1pt] (90:1+\width) -- (0,0);
            \end{scope}

            \fill[blue] (0,0) circle (0.05cm);
            \fill[blue] (90:1+\width) circle (0.05cm);
            \fill[blue] (210:1+\width) circle (0.05cm);
            \fill[blue] (330:1+\width) circle (0.05cm);

            \def\edgeA{black}\def\edgeB{black}\def\edgeC{black}\def\edgeD{black}\def\edgeE{black}\def\edgeF{black}
            \edges
        \end{tikzpicture}
        \caption{An embedding $K_4\into\Sigma$ mapping the vertices to the `extremities' of $\Sigma$. Every cycle in $K_4$ is associated with a curve $\gamma_0$ under this embedding, which we choose to be the reference curves for which every other curve $\gamma\in X_0$ will be a multitwist thereof.}
        \label{fig:embedding_graph_in_surface}
    \end{figure}

    Suppose that $\gamma$ intersects exactly three pairs of pants $\mc{P}_0^\gamma\subset\mc{P}_0$. Under the action $A_4\act\Sigma$, we can assume that the pants curves of $\mc{P}_0^\gamma$ are $\alpha_1,\dots,\alpha_3$. The dual graph of $\mc{P}_0^\gamma$ is then a $3$-cycle $C\subset K_4$ $-$ whose edges we label by $\alpha_1,\dots,\alpha_3$ $-$ which determines a curve $\gamma_0$ under the embedding $K_4\into\Sigma$. For each of the three pairs of pants, $\gamma_0$ projects onto the unique arc joining two of its boundary curves in $\alpha_1,\dots,\alpha_3$. Since $i(\gamma,\alpha_j)\leq1$ for $j=1,2,3$, we see that $\gamma$ projects onto a single copy of those arcs, so $\gamma$ is a multitwist of $\gamma_0$ in $\alpha_j$. Thus $\gamma=T_{\alpha_1}^{\epsilon_1}T_{\alpha_2}^{\epsilon_2}T_{\alpha_3}^{\epsilon_3}(\gamma_0)$ for some integers $\epsilon_j\in\Z$, which we call the \textit{twist parameters} of $\gamma$. Fixing an ordering of $\alpha_j$ gives us an encoding of $\gamma$ as a tuple $\tpl{\epsilon_1,\epsilon_2,\epsilon_3}$, which we refer to as a \textit{curve of type $\triangle$ on $C$}. See Figure \ref{fig:encoding_triangle} for an illustration of the curve encoded by $\tpl{1,1,0}$.

    \begin{figure}[b]
        \begin{tikzpicture}
            \begin{scope}
                \shell
                \draw[blue] plot [smooth cycle, tension=0.4] coordinates {(90:1+\width/2) (210:1+\width/2) (330:1+\width/2)};

                \def\edgeA{black}\def\edgeB{black}\def\edgeC{black}
                \edges
            \end{scope}
            \draw[->] (2,0.3) -- node[above]{\small$T_{\alpha_1}^1T_{\alpha_2}^1$} (3.5,0.3);
            \begin{scope}[xshift=5.5cm]
                \shell
                \draw[blue] plot [smooth cycle, tension=0.4] coordinates {(90:1+\width/2) (210:1+\width/2) (330:1+\width/2)};
                \def\twistColor{blue}\def\twistStart{0.904}\def\twistEnd{0.904}\def\twistPosition{0}\twistFront
                \def\twistColor{blue}\def\twistStart{0.904}\def\twistEnd{0.904}\def\twistPosition{120}\twistFront

                \def\edgeA{black}\def\edgeB{black}\def\edgeC{black}
                \edges
            \end{scope}
        \end{tikzpicture}
        \caption{A reference curve $\gamma_0$ and its multitwist $\gamma=T_{\alpha_1}^1T_{\alpha_2}^1(\gamma_0)$, encoded by $\tpl{1,1,0}$.}
        \label{fig:encoding_triangle}
    \end{figure}

    Now, if $\gamma$ intersects all four pairs of pants in $\mc{P}_0$, then the curves $\alpha_{j_1},\dots,\alpha_{j_4}$ that it intersects determine a $4$-cycle in $K_4$ and hence a reference curve $\gamma_0$ under the embedding $K_4\into\Sigma_3$. The condition $i(\gamma,\alpha_{j_l})\leq1$ for $l=1,\dots,4$ then similarly makes $\gamma$ a multitwist of $\gamma_0$ in $\alpha_{j_l}$, so $\gamma=T_{\alpha_{j_1}}^{\epsilon_{j_1}}\cdots T_{\alpha_{j_4}}^{\epsilon_{j_4}}(\gamma_0)$ for some integers $\epsilon_{j_l}\in\Z$, which we again call the \textit{twist parameters of $\gamma$}. Fixing an ordering on $\alpha_1,\dots,\alpha_6$ allows us to insert two `blanks' in the tuple $\tpl{\epsilon_{j_1},\dots,\epsilon_{j_4}}$ to represent the two edges of $K_4$ not traversed by $\gamma$; the $6$-tuple is then referred to as a \textit{curve of type $\square$}. See Figure \ref{fig:encoding_square} for an illustration of the curve encoded by $\tpl{1,\slot,0,\slot,0,1}$.

    \begin{figure}[b]
        \begin{tikzpicture}
            \begin{scope}
                \shell

                \begin{scope}
                    \draw[blue] (150:0.89) .. controls (210:1+\width-0.1) .. (210:1+\width);
                \end{scope}
                \begin{scope}[rotate=240]
                    \draw[blue] (90:1+\width) .. controls (90:1+\width-0.1) .. (150:0.89);
                \end{scope}

                \def\curve{0.08}
                \draw[blue, dashed, dash pattern={on 2pt off 1pt}] (210:1+\width) .. controls (210:1+\width-0.1) .. (210:\width/2) .. controls (210:\width/2-0.1) and (180:\curve) .. (0,0) .. controls (0:\curve) and (330:\width/2-0.1) .. (330:\width/2) .. controls (330:1+\width-0.1) .. (330:1+\width);
                \begin{scope}
                    \clip (-0.78,0.4) rectangle (0.78,1.44);
                    \draw[blue] plot [smooth cycle, tension=0.4] coordinates {(90:1+\width/2) (210:1+\width/2) (330:1+\width/2)};
                \end{scope}

                \def\edgeA{black}\def\edgeB{gray}\def\edgeC{black}\def\edgeD{gray}\def\edgeE{black}\def\edgeF{black}
                \edges
            \end{scope}
            \draw[->] (2,0.3) -- node[above]{\small$T_{\alpha_1}^1T_{\alpha_6}^1$} (3.5,0.3);
            \begin{scope}[xshift=5.5cm]
                \shell

                \begin{scope}
                    \draw[blue] (150:0.89) .. controls (210:1+\width-0.1) .. (210:1+\width);
                    \def\twistColor{blue}\def\twistStart{0.904}\def\twistEnd{0.904}\def\twistPosition{0}\twistFront
                \end{scope}
                \begin{scope}[rotate=240]
                    \draw[blue] (90:1+\width) .. controls (90:1+\width-0.1) .. (150:0.89);
                \end{scope}

                \def\curve{0.08}

                \draw[blue, dashed, dash pattern={on 2pt off 1pt}] (210:1+\width) .. controls (210:1+\width-0.1) .. (210:\width/2) .. controls (210:\width/2-0.1) and (180:\curve) .. (0,0) .. controls (0:\curve) and (330:\width/2-0.1) .. (330:\width/2) .. controls (330:1+\width-0.1) .. (330:1+\width);

                \begin{scope}
                    \clip (-0.695,0.4) rectangle (0.78,1.44);
                    \draw[blue] plot [smooth cycle, tension=0.4] coordinates {(90:1+\width/2) (210:1+\width/2) (330:1+\width/2)};
                \end{scope}

                \def\twistColor{blue}\def\twistStart{0.5}\def\twistEnd{0}\def\twistPosition{240}\def\twistExtend{0.02}\def\twistShift{0}\twistBack

                \def\edgeA{black}\def\edgeB{gray}\def\edgeC{black}\def\edgeD{gray}\def\edgeE{black}\def\edgeF{black}
                \edges
            \end{scope}
        \end{tikzpicture}
        \caption{A reference curve $\gamma_0$ and its multitwist $\gamma=T_{\alpha_1}^1T_{\alpha_6}^1(\gamma_0)$, encoded by $\tpl{1,\slot,0,\slot,0,1}$.}
        \label{fig:encoding_square}
    \end{figure}

    Using this encoding, we state our main theorem. It describes a construction of a saturated $1$-system $X_0$ on $\Sigma$ consisting of $6$ pairwise disjoint curves $\alpha_j$, $24$ curves of type $\triangle$, and $3$ curves of type $\square$.

    \begin{theorem}\label{thm:construction_of_1_system}
        Let $\Sigma\coloneqq\Sigma_3$ be the surface of genus $3$ and consider the $6$ pairwise disjoint curves $\alpha_1,\dots,\alpha_6$ on $\Sigma$, obtained as boundary curves of the pairs of pants in the pants decomposition $\mc{P}_0$ of $\Sigma$ and distributed as in the preceding figures so that, in particular, $\alpha_1,\dots,\alpha_3$ induces a $3$-cycle $C\subset K_4$. Let $\gamma_1\coloneqq\tpl{1,0,0}$ and $\gamma_2\coloneqq\tpl{1,1,0}$ be curves of type $\triangle$ on $C$, and let $\gamma_3\coloneqq\tpl{1,\slot,0,\slot,0,1}$ be a curve of type $\square$.

        Consider the system of curves $X_0\coloneqq\l\{\alpha_1,\dots,\alpha_6\r\}\cup\l\{G\gamma_j\st j=1,2,3\r\}$ on $\Sigma$, where $G\gamma_j$ is the orbit of $\gamma_j$ under the group $G\coloneqq A_4$ acting on $\mc{C}_0(\Sigma)$ by orientation-preserving symmetries of $\Sigma$. Then $X_0$ is a saturated $1$-system of $33$ curves on $\Sigma$.
    \end{theorem}

    The proof of this theorem will occupy the next two sections. First, we analyze the curves of types $\triangle$ and $\square$ in $X_0$ separately in the next section, and in particular prove that the collection of curves of type $\triangle$ in $X_0$ is saturated amongst \textit{all} curves of type $\triangle$ on $\Sigma$. The corresponding saturation statement for curves of type $\square$ is slightly harder and will be deferred to the proof of the main theorem.

    \subsection{Curves of types $\triangle$ and $\square$}\label{sec:curves_of_types_triangle_and_square}

    We consider the curves of types $\triangle$ and $\square$ in $X_0$ separately and prove that they are individually $1$-systems. First, we need a lemma on the intersection numbers of arcs; it is applicable when cutting $\Sigma$ projects curves onto arcs in such a way that does not modify their intersection numbers.
    \begin{lemma}\label{lem:arcs_dehn_twists}
        Identify the $4$-holed sphere $\Sigma_{0,4}$ as two pairs of pants $P_1$ and $P_2$ glued along some curve $\mu$ and fix four arcs $\alpha_0$, $\beta_0$, $\gamma_{0.5}$, and $\delta_{0.5}$ between the other boundary curves of $P_1$ and $P_2$, as shown in Figure \ref{fig:arcs_dehn_twists}.
        \begin{figure}[h]\hspace{0.2in}
            \begin{minipage}{0.25\textwidth}
                \begin{tikzpicture}
                    \def\ellipseHor{1.75}
                    \def\ellipseVer{1}
                    \def\rad{0.3}
                    \def\horSep{0.6}
                    \def\verSep{2}

                    \draw (0,0) circle [x radius=\ellipseHor, y radius=\ellipseVer];
                    \draw (-\horSep,0) circle (\rad);
                    \draw (\horSep,0) circle (\rad);
                    \draw (-\horSep,-\verSep) circle (\rad);
                    \draw (\horSep,-\verSep) circle (\rad);

                    \node at (\ellipseHor-0.3,0) {\small$P_1$};
                    \node at (\ellipseHor-0.3,-\verSep) {\small$P_2$};
                    \node at (\ellipseHor+0.2,0) {\small$\mu$};

                    \draw[thick] (-\horSep,-\rad) -- node[below right=-0.05cm]{\small$\alpha_0$} (-\horSep,\rad-\verSep);
                    \draw[thick] (\horSep,-\rad) -- node[below right=-0.05cm]{\small$\beta_0$} (\horSep,\rad-\verSep);
                    \draw[thick] (-\horSep+0.05,-\rad) .. controls (-\horSep+0.05,-\rad-0.05) and (\horSep-0.05,\rad-\verSep+0.05) .. node[above=0.15cm]{\small$\gamma_{0.5}$} (\horSep-0.05,\rad-\verSep);
                    \draw[thick] (-\horSep-0.05,\rad-\verSep) .. controls (-\horSep-0.06,-\verSep/2) and (-\ellipseHor-0.1,-\ellipseVer+0.1) .. (-\ellipseHor-0.1,0) node[below left=-0.1cm]{\small$\delta_{0.5}$} .. controls (-\ellipseHor-0.1,\ellipseVer+0.1) and (\horSep,\verSep/2) .. (\horSep,\rad);
                \end{tikzpicture}
            \end{minipage}
            \begin{minipage}{0.7\textwidth}
                \caption{Four pairwise disjoint arcs $\alpha_0$, $\beta_0$, $\gamma_{0.5}$, and $\delta_{0.5}$.}
                \label{fig:arcs_dehn_twists}
            \end{minipage}
        \end{figure}

        For each $n\in\Z$, let $\alpha_n\coloneqq T_\mu^n(\alpha_0)$, and similarly define $\beta_n$, $\gamma_{n+0.5}$, and $\delta_{n+0.5}$. If $\ast$ and $\ast'$ are distinct symbols denoting two of $\alpha$, $\beta$, $\gamma$, and $\delta$, then $i(\ast_n,\ast_m')=\l\lfloor\l|n-m\r|\r\rfloor$ for all integers and half-integers $m,n$.
    \end{lemma}
    \begin{proof}
        By symmetry, it suffices to prove the statement for the cases when $\tpl{\ast,\ast'}=\tpl{\alpha,\beta}$ and $\tpl{\ast,\ast'}=\tpl{\alpha,\gamma}$. Let $n\in\Z$ and consider the arc $\alpha_n$.

        For an arc $\beta_m$ with $m\in\Z$, we can simultaneously undo twists so that $m=0$. Every twist of $\alpha_0$ in $\mu$ increases its intersection with $\beta_0$ by $1$, and since no bigons or half-bigons are created in this process (Figure \ref{fig:arc_dehn_twist_1}), we see from the Bigon Criterion that $i(\alpha_n,\beta_0)=\l|n\r|$.

        For an arc $\gamma_m$ with $m\in\Z+0.5$, we can simultaneously undo twists so that either $m>n=0$ or $n>m=0.5$. Using the Bigon Criterion, we see in the former case (Figure \ref{fig:arc_dehn_twist_2}) that $i(\alpha_0,\gamma_m)=m-0.5=\l\lfloor m-n\r\rfloor$. Otherwise (Figure \ref{fig:arc_dehn_twist_3}), the arcs $\alpha_n$ and $\gamma_{0.5}$ bound a half-bigon and hence can be homotoped to decrease the intersection by $1$, so $i\,(\alpha_n,\gamma_{0.5})=n-1=\l\lfloor n-m\r\rfloor$.\qed
        \begin{figure}
            \def\ellipseHor{1.75}\def\ellipseVer{1}\def\rad{0.3}\def\horSep{0.6}\def\verSep{2}
            \begin{minipage}[t]{0.3\textwidth}
                \begin{center}
                    \begin{tikzpicture}[scale=0.8]
                        \draw (0,0) circle [x radius=\ellipseHor, y radius=\ellipseVer];
                        \draw (-\horSep,0) circle (\rad);
                        \draw (\horSep,0) circle (\rad);
                        \draw (-\horSep,-\verSep) circle (\rad);
                        \draw (\horSep,-\verSep) circle (\rad);

                        \draw[thick]
                            (-\horSep,-\rad)
                                .. controls (-\horSep,-\ellipseVer) and (\ellipseHor-0.2,-\ellipseVer) ..
                            (\ellipseHor-0.2,0)
                                .. controls (\ellipseHor-0.2,\ellipseVer-0.5) and (\ellipseHor-0.8,\ellipseVer-0.2) ..
                            (0,\ellipseVer-0.2)
                                .. controls (-\ellipseHor+0.8,\ellipseVer-0.2) and (-\ellipseHor+0.2,\ellipseVer-0.5) ..
                            (-\ellipseHor+0.2,0)
                                .. controls (-\ellipseHor+0.2,-\ellipseVer+0.1) and (-\horSep,-\verSep/2) ..
                            (-\horSep,\rad-\verSep) node[xshift=-0.3cm, yshift=0.2cm]{\small$\alpha_1$};
                        \draw[thick] (\horSep,-\rad) -- node[below right=-0.05cm]{\small$\beta_0$} (\horSep,\rad-\verSep);
                    \end{tikzpicture}
                    \subcaption{Neither bigons nor half-bigons are created by twisting $\alpha_0$ in $\mu$.}
                    \label{fig:arc_dehn_twist_1}
                \end{center}
            \end{minipage}\hspace{0.2in}
            \begin{minipage}[t]{0.3\textwidth}
                \begin{center}
                    \begin{tikzpicture}[scale=0.8]
                        \draw (0,0) circle [x radius=\ellipseHor, y radius=\ellipseVer];
                        \draw (-\horSep,0) circle (\rad);
                        \draw (\horSep,0) circle (\rad);
                        \draw (-\horSep,-\verSep) circle (\rad);
                        \draw (\horSep,-\verSep) circle (\rad);

                        \draw[thick] (-\horSep,-\rad) -- node[below left=-0.05cm]{\small$\alpha_0$} (-\horSep,\rad-\verSep);
                        \draw[thick]
                            (-\horSep+0.05,-\rad)
                                .. controls (-\horSep+0.05,-\ellipseVer) and (\ellipseHor-0.2,-\ellipseVer) ..
                            (\ellipseHor-0.2,0)
                                .. controls (\ellipseHor-0.2,\ellipseVer-0.5) and (\ellipseHor-0.8,\ellipseVer-0.2) ..
                            (0,\ellipseVer-0.2)
                                .. controls (-\ellipseHor+0.8,\ellipseVer-0.2) and (-\ellipseHor+0.2,\ellipseVer-0.5) ..
                            (-\ellipseHor+0.2,0)
                                .. controls (-\ellipseHor+0.2,-\ellipseVer+0.1) and (\horSep-0.06,-\verSep/2) .. node[xshift=1.2cm, yshift=-0.3cm]{\small$\gamma_{1+0.5}$}
                            (\horSep-0.05,\rad-\verSep);
                    \end{tikzpicture}
                    \subcaption{Neither bigons nor half-bigons are created by twisting $\gamma_{0.5}$ in $\mu$.}
                    \label{fig:arc_dehn_twist_2}
                \end{center}
            \end{minipage}\hspace{0.2in}
            \begin{minipage}[t]{0.3\textwidth}
                \begin{center}
                    \begin{tikzpicture}[scale=0.8]
                        \draw (0,0) circle [x radius=\ellipseHor, y radius=\ellipseVer];
                        \draw (-\horSep,0) circle (\rad);
                        \draw (\horSep,0) circle (\rad);
                        \draw (-\horSep,-\verSep) circle (\rad);
                        \draw (\horSep,-\verSep) circle (\rad);

                        \draw[thick]
                            (-\horSep,-\rad)
                                .. controls (-\horSep,-\ellipseVer) and (\ellipseHor-0.2,-\ellipseVer) ..
                            (\ellipseHor-0.2,0)
                                .. controls (\ellipseHor-0.2,\ellipseVer-0.5) and (\ellipseHor-0.8,\ellipseVer-0.2) ..
                            (0,\ellipseVer-0.2)
                                .. controls (-\ellipseHor+0.8,\ellipseVer-0.2) and (-\ellipseHor+0.2,\ellipseVer-0.5) ..
                            (-\ellipseHor+0.2,0)
                                .. controls (-\ellipseHor+0.2,-\ellipseVer+0.1) and (-\horSep,-\verSep/2) ..
                            (-\horSep,\rad-\verSep) node[xshift=-0.3cm, yshift=0.2cm]{\small$\alpha_1$};
                        \draw[thick]
                            (-\horSep+0.05,-\rad)
                            .. controls (-\horSep+0.05,-\rad-0.05) and (\horSep+0.05,\rad-\verSep+0.05) .. node[xshift=0.5cm, yshift=-0.2cm]{\small$\gamma_{0.5}$}
                            (\horSep+0.05,\rad-\verSep);
                    \end{tikzpicture}
                    \subcaption{A half-bigon is created between $\alpha_1$ and $\gamma_{0.5}$ by twisting $\alpha_0$ in $\mu$.}
                    \label{fig:arc_dehn_twist_3}
                \end{center}
            \end{minipage}
            \caption{Twisting a curve in $\mu$ generally increases its intersection with the other curve by $1$. The Bigon Criterion corrects this estimate by accounting for any bigons and half-bigons created.}
        \end{figure}
    \end{proof}

    \begin{proposition}\label{prp:type_triangle_24}
        In the notation of Theorem \ref{thm:construction_of_1_system}, the union of the orbits of $\gamma_1=\tpl{1,0,0}$ and $\gamma_2=\tpl{1,1,0}$ under $G$ is a $1$-system consisting of $24$ curves of type $\triangle$ that it is saturated amongst all curves of type $\triangle$.
    \end{proposition}
    \begin{proof}
        Let $v\in K_4$ be the vertex opposite to $C$. Then the union of the orbits of $\gamma_1$ and $\gamma_2$ under the action of the stabilizer subgroup $\Stab_G(v)\iso\Z/3\Z$ is the collection
        \begin{equation*}
            X_C\coloneqq\l\{\tpl{1,0,0},\tpl{0,1,0},\tpl{0,0,1},\tpl{1,1,0},\tpl{1,0,1},\tpl{0,1,1}\r\},
        \end{equation*}
        which we claim is a $1$-system. Indeed, take any two curves $\gamma,\gamma'\in X_C$. Modulo the action of $\Stab_G(v)$ and simultaneously undoing twists $-$ which do not change $i(\gamma,\gamma')$ $-$ we may assume that $\gamma=\tpl{1,0,0}$ and that $\gamma'$ is one of $\tpl{0,0,0}$, $\tpl{0,1,0}$, and $\tpl{0,1,1}$. Note that the case where $\gamma'=\tpl{0,0,0}$ arises, for instance, when we consider the curves $\tpl{1,1,0}$ and $\tpl{0,1,0}$; simultaneously undoing the twist in $\alpha_2$ gives us $\gamma=\tpl{1,0,0}$ and $\gamma'=\tpl{0,0,0}$, so this case must be considered even if $\tpl{0,0,0}\not\in X_C$.
        \begin{figure}
            \begin{minipage}[t]{0.3\textwidth}
                \begin{center}
                    \begin{tikzpicture}
                        \shell

                        \draw[blue] plot [smooth cycle, tension=0.4] coordinates {(90:1+\width/2) (210:1+\width/2) (330:1+\width/2)};
                        \def\twistColor{blue}\def\twistStart{0.904}\def\twistEnd{0.904}\def\twistPosition{0}\twistFront

                        \draw[darkGreen] plot [smooth cycle, tension=0.4] coordinates {(90:1+\width/2+0.1) (210:1+\width/2+0.1) (330:1+\width/2+0.1)};

                        \def\edgeA{black}\def\edgeB{black}\def\edgeC{black}
                        \edges

                        \fill (155:0.955) circle (0.01in);
                    \end{tikzpicture}
                    \subcaption{$\gamma'=\tpl{0,0,0}$.}
                \end{center}
            \end{minipage}
            \begin{minipage}[t]{0.3\textwidth}
                \begin{center}
                    \begin{tikzpicture}
                        \shell

                        \draw[darkGreen] plot [smooth cycle, tension=0.4] coordinates {(90:1+\width/2+0.1) (210:1+\width/2+0.1) (330:1+\width/2+0.1)};
                        \draw[darkGreen] plot [smooth cycle, tension=0.4] coordinates {(90:1+\width/2-0.1) (210:1+\width/2-0.1) (330:1+\width/2-0.1)};
                        \draw[ultra thick, white] (145:0.955) to [out=240, in=80] (210:1.55);
                        \draw[ultra thick, white] (210:1.55) to [out=-20, in=180] (265:0.95);
                        \begin{scope}[rotate=120]
                            \draw[ultra thick, white] (145:0.84) to [out=240, in=80] (210:1.35);
                            \draw[ultra thick, white] (210:1.35) to [out=-20, in=180] (265:0.84);
                        \end{scope}
                        \begin{scope}[rotate=240]
                            \draw[ultra thick, white] (145:0.84) to [out=240, in=80] (210:1.35);
                            \draw[ultra thick, white] (210:1.35) to [out=-20, in=180] (280:0.84);
                        \end{scope}
                        \def\twistColor{darkGreen}\def\twistStart{0.84}\def\twistEnd{0.9675}\def\twistPosition{120}\twistFront

                        \draw[blue] plot [smooth cycle, tension=0.4] coordinates {(90:1+\width/2) (210:1+\width/2) (330:1+\width/2)};
                        \def\twistColor{blue}\def\twistStart{0.904}\def\twistEnd{0.904}\def\twistPosition{0}\twistFront

                        \draw[darkGreen] (144:0.958) to [out=240, in=60] (161:0.842);

                        \def\edgeA{black}\def\edgeB{black}\def\edgeC{black}
                        \edges
                    \end{tikzpicture}
                    \subcaption{$\gamma'=\tpl{0,1,0}$.}
                \end{center}
            \end{minipage}
            \begin{minipage}[t]{0.3\textwidth}
                \begin{center}
                    \begin{tikzpicture}
                        \shell

                        \draw[darkGreen] plot [smooth cycle, tension=0.4] coordinates {(90:1+\width/2+0.1) (210:1+\width/2+0.1) (330:1+\width/2+0.1)};
                        \draw[darkGreen] plot [smooth cycle, tension=0.4] coordinates {(90:1+\width/2-0.1) (210:1+\width/2-0.1) (330:1+\width/2-0.1)};
                        \draw[ultra thick, white] (145:0.955) to [out=240, in=80] (210:1.55);
                        \draw[ultra thick, white] (210:1.55) to [out=-20, in=180] (265:0.95);
                        \begin{scope}[rotate=120]
                            \draw[ultra thick, white] (145:0.84) to [out=240, in=80] (210:1.35);
                            \draw[ultra thick, white] (210:1.35) to [out=-20, in=180] (265:0.84);
                        \end{scope}
                        \begin{scope}[rotate=240]
                            \draw[ultra thick, white] (145:0.84) to [out=240, in=80] (210:1.35);
                            \draw[ultra thick, white] (210:1.35) to [out=-20, in=180] (280:0.84);
                        \end{scope}
                        \def\twistColor{darkGreen}\def\twistStart{0.84}\def\twistEnd{0.9675}\def\twistPosition{120}\twistFront
                        \def\twistColor{darkGreen}\def\twistStart{0.9675}\def\twistEnd{0.9675}\def\twistPosition{240}\twistFront

                        \draw[blue] plot [smooth cycle, tension=0.4] coordinates {(90:1+\width/2) (210:1+\width/2) (330:1+\width/2)};
                        \def\twistColor{blue}\def\twistStart{0.904}\def\twistEnd{0.904}\def\twistPosition{0}\twistFront

                        \draw[darkGreen] (144:0.958) to [out=240, in=60] (161:0.842);

                        \def\edgeA{black}\def\edgeB{black}\def\edgeC{black}
                        \edges

                        \fill (25:0.894) circle (0.01in);
                    \end{tikzpicture}
                    \subcaption{$\gamma'=\tpl{0,1,1}$.}
                \end{center}
            \end{minipage}
            \caption{The curves {\color{blue}$\gamma=\tpl{1,0,0}$} and {\color{darkGreen}$\gamma'\in\l\{\tpl{0,0,0},\tpl{0,1,0},\tpl{0,1,1}\r\}$} of type $\triangle$ on $C$, which intersect at most once in all three cases, so $X_C$ is a $1$-system.}
            \label{fig:type_triangle_curves}
        \end{figure}
        Figure \ref{fig:type_triangle_curves} shows that the union of the orbits of $\gamma_1$ and $\gamma_2$ under $\Stab_G(v)$ is a $1$-system consisting of $6$ curves of type $\triangle$ on $C$.

        The entire union $X_\triangle$ of the orbits of $\gamma_1$ and $\gamma_2$ under $G$ is then partitioned into four $1$-systems on the $3$-cycles of the tetrahedron, each with $6$ curves. We show that this collection $X_\triangle$ of $24$ curves of type $\triangle$ remains a $1$-system, so take $\gamma,\gamma'\in X_\triangle$. If they are on the same face, then we are done by the above. Otherwise, they share exactly one edge of $K_4$, which joins two pairs of pants $P_1$ and $P_2$ in the pants decomposition $\mc{P}_0$ of $\Sigma$. Glueing $P_1$ and $P_2$ together gives us a subsurface of $\Sigma$ making $\pi(\gamma)$ and $\pi(\gamma')$ a pair of arcs in the situation of Lemma \ref{lem:arcs_dehn_twists}, from which it follows that $i(\gamma,\gamma')=i(\pi(\gamma),\pi(\gamma'))\leq1$.

        Finally, we verify\footnote{We thank the referee for the following elegant proof of this statement.} that this $1$-system is saturated amongst curves of type $\triangle$. Suppose towards a contradiction that there are more than $24$ curves on type $\triangle$, so there is a $3$-cycle in $K_4$ supporting at least $7$ curves of type $\triangle$. But then there would be a $1$-system of $7+3=10$ curves on the subsurface $\Sigma_{1,3}$ corresponding to $C$, which is impossible since $N(1,1,3)=9$ by \cite{MRT14}*{Theorem 1.2}.
    \end{proof}

    \begin{notation*}
        For any $3$-cycle $C'\subset K_4$, let $X_{C'}$ denote the $1$-system consisting of $6$ curves of type $\triangle$ on $C'$.
    \end{notation*}

    \begin{proposition}\label{prp:type_square_3}
        In the notation of Theorem \ref{thm:construction_of_1_system}, the orbit of $\gamma_3=\tpl{1,\slot,0,\slot,0,1}$ under $G$ is a $1$-system consisting of $3$ curves of type $\square$.
    \end{proposition}
    \begin{proof}
        The group $G=A_4$ acts on the vertices $\l\{v_1,\dots,v_4\r\}$ of $K_4$ by $\sigma v_j\coloneqq v_{\sigma(j)}$, and it can easily be seen from Figure \ref{fig:tetrahedral_square_curve} that $\gamma_3$ is fixed under the Klein $4$ subgroup $K\leq G$. Note that $G\comp K$ is partitioned into $S\coloneqq\l\{(1\ 2\ 3),(1\ 4\ 2),(1\ 3\ 4),(2\ 4\ 3)\r\}$ and their inverses $S^{-1}$, and again from Figure \ref{fig:tetrahedral_square_curve}, it can be easily verified that $S\gamma_3=\l\{\tpl{\slot,0,1,0,1,\slot}\r\}\eqqcolon\l\{\gamma_3'\r\}$ and $S^{-1}\gamma_3=\l\{\tpl{0,1,\slot,1,\slot,0}\r\}\eqqcolon\l\{\gamma_3''\r\}$.
        \begin{figure}
            \begin{minipage}[t]{0.3\textwidth}
                \begin{center}
                    \begin{tikzpicture}
                        \shell

                        \begin{scope}
                            \begin{scope}
                                \draw[darkGreen] (150:0.89) .. controls (210:1+\width-0.1) .. (210:1+\width);
                                \def\twistColor{darkGreen}\def\twistStart{0.904}\def\twistEnd{0.904}\def\twistPosition{0}\twistFront
                            \end{scope}
                            \begin{scope}[rotate=240]
                                \draw[darkGreen] (90:1+\width) .. controls (90:1+\width-0.1) .. (150:0.89);
                            \end{scope}

                            \def\curve{0.08}

                            \draw[darkGreen, dashed, dash pattern={on 2pt off 1pt}] (210:1+\width) .. controls (210:1+\width-0.1) .. (210:\width/2) .. controls (210:\width/2-0.1) and (180:\curve) .. (0,0) .. controls (0:\curve) and (330:\width/2-0.1) .. (330:\width/2) .. controls (330:1+\width-0.1) .. (330:1+\width);

                            \def\twistColor{darkGreen}\def\twistPosition{240}\def\twistExtend{0}\def\twistShift{0}\twistBack

                            \begin{scope}
                                \clip (-0.695,0.4) rectangle (0.78,1.44);
                                \draw[darkGreen] plot [smooth cycle, tension=0.4] coordinates {(90:1+\width/2) (210:1+\width/2) (330:1+\width/2)};
                            \end{scope}
                        \end{scope}

                        \def\edgeA{black}\def\edgeB{white}\def\edgeC{black}\def\edgeE{black}\def\edgeF{black}
                        \edges

                        \vertices
                    \end{tikzpicture}
                    \subcaption{$\gamma_3=\tpl{1,\slot,0,\slot,0,1}$.}
                \end{center}
            \end{minipage}
            \begin{minipage}[t]{0.3\textwidth}
                \begin{center}
                    \begin{tikzpicture}
                        \shell

                        \begin{scope}[rotate=240]
                            \begin{scope}
                                \draw[red] (150:0.89) .. controls (210:1+\width-0.1) .. (210:1+\width);
                                \def\twistColor{red}\def\twistStart{0.904}\def\twistEnd{0.904}\def\twistPosition{0}\twistFront
                            \end{scope}
                            \begin{scope}[rotate=240]
                                \draw[red] (90:1+\width) .. controls (90:1+\width-0.1) .. (150:0.89);
                            \end{scope}

                            \def\curve{0.08}

                            \draw[red, dashed, dash pattern={on 2pt off 1pt}] (210:1+\width) .. controls (210:1+\width-0.1) .. (210:\width/2) .. controls (210:\width/2-0.1) and (180:\curve) .. (0,0) .. controls (0:\curve) and (330:\width/2-0.1) .. (330:\width/2) .. controls (330:1+\width-0.1) .. (330:1+\width);

                            \def\twistColor{red}\def\twistPosition{240}\def\twistExtend{0.02}\def\twistShift{0}\twistBack

                            \begin{scope}
                                \clip (-0.695,0.4) rectangle (0.78,1.44);
                                \draw[red] plot [smooth cycle, tension=0.4] coordinates {(90:1+\width/2) (210:1+\width/2) (330:1+\width/2)};
                            \end{scope}
                        \end{scope}

                        \def\edgeB{black}\def\edgeC{black}\def\edgeD{black}\def\edgeE{black}
                        \edges

                        \vertices
                    \end{tikzpicture}
                    \subcaption{$\gamma_3'=\tpl{\slot,0,1,0,1,\slot}$.}
                \end{center}
            \end{minipage}
            \begin{minipage}[t]{0.3\textwidth}
                \begin{center}
                    \begin{tikzpicture}
                        \shell

                        \begin{scope}[rotate=120]
                            \begin{scope}
                                \draw[blue] (150:0.89) .. controls (210:1+\width-0.1) .. (210:1+\width);
                                \def\twistColor{blue}\def\twistStart{0.904}\def\twistEnd{0.904}\def\twistPosition{0}\twistFront
                            \end{scope}
                            \begin{scope}[rotate=240]
                                \draw[blue] (90:1+\width) .. controls (90:1+\width-0.1) .. (150:0.89);
                            \end{scope}

                            \def\curve{0.08}

                            \draw[blue, dashed, dash pattern={on 2pt off 1pt}] (210:1+\width) .. controls (210:1+\width-0.1) .. (210:\width/2) .. controls (210:\width/2-0.1) and (180:\curve) .. (0,0) .. controls (0:\curve) and (330:\width/2-0.1) .. (330:\width/2) .. controls (330:1+\width-0.1) .. (330:1+\width);

                            \def\twistColor{blue}\def\twistPosition{240}\def\twistExtend{0.02}\def\twistShift{0}\twistBack

                            \begin{scope}
                                \clip (-0.695,0.4) rectangle (0.78,1.44);
                                \draw[blue] plot [smooth cycle, tension=0.4] coordinates {(90:1+\width/2) (210:1+\width/2) (330:1+\width/2)};
                            \end{scope}
                        \end{scope}

                        \def\edgeA{black}\def\edgeB{black}\def\edgeD{black}\def\edgeF{black}
                        \edges

                        \vertices
                    \end{tikzpicture}
                    \subcaption{$\gamma_3''=\tpl{0,1,\slot,1,\slot,0}$.}
                \end{center}
            \end{minipage}
            \caption{\small The orbit of $\gamma_3$ under $G$ is a collection of three pairwise disjoint curves of type $\square$.}
            \label{fig:tetrahedral_square_curve}
        \end{figure}
        Thus $X_\square\coloneqq G\gamma_3=\l\{\gamma_3,\gamma_3',\gamma_3''\r\}$ consists of three curves, so it remains to show that they pairwise intersect at most once.

        Under the action of $G$, it suffices to show that $i(\gamma_3,\gamma_3')\leq1$. Indeed, note that $\gamma_3$ and $\gamma_3'$ only share two edges, and cutting $\Sigma$ along the other $4$ disjoint curves $\alpha_j$ decomposes it into two subsurfaces. This process does not introduce any half-bigons, so the intersection number $i(\gamma_3,\gamma_3')$ does not decrease. The arcs projected from $\gamma_3$ and $\gamma_3'$ onto any one of the subsurfaces $-$ which is two pairs of pants glued along a boundary curve $-$ are in the situation of Lemma \ref{lem:arcs_dehn_twists}, from which it follows that $i(\pi(\gamma_3),\pi(\gamma_3'))=\l\lfloor\frac{1}{2}-\frac{1}{2}\r\rfloor=0$. The same holds for the other subsurface, and since they do not intersect elsewhere, the result follows.\qed
    \end{proof}

    \subsection{Proof of Main Theorem}

    We collect the results from Section \ref{sec:curves_of_types_triangle_and_square} to prove that the system of curves constructed in Theorem \ref{thm:construction_of_1_system} is indeed a saturated $1$-system. To this end, we need the following lemma, which limits the number of curves of type $\square$ that can be included in $X_0$.

    \begin{lemma}\label{lem:exactly_one_type_square}
        In the notation of Theorem \ref{thm:construction_of_1_system}, cut the subsurface $\Sigma_{1,3}$ corresponding to $C$ along $\alpha_2$ to obtain a sphere $\Sigma_{0,5}$ with $5$ holes (see Figure \ref{fig:exactly_one_type_square_arcs}). Then, within $\Sigma_{0,5}$, there is exactly one arc $\eta$ joining $\alpha_5$ and $\alpha_6$ making $\pi(X_C)\cup\l\{\eta\r\}$ a $1$-system of arcs.

        \begin{figure}[h]\hspace{0.1in}
            \begin{tikzpicture}
                \shell

                \begin{scope}
                    \draw[darkGreen] (150:0.89) .. controls (210:1+\width-0.1) .. (210:1+\width);
                    \def\twistColor{darkGreen}\def\twistStart{0.904}\def\twistEnd{0.904}\def\twistPosition{0}\twistFront
                \end{scope}
                \begin{scope}[rotate=240]
                    \draw[darkGreen] (90:1+\width) .. controls (90:1+\width-0.1) .. (150:0.89);
                \end{scope}

                \fill[white] (-0.3,-0.5) rectangle (0.3,0.545);
                \begin{scope}[rotate=120]
                    \fill[white] (-0.3,-0.5) rectangle (0.3,0.545);
                \end{scope}
                \begin{scope}[rotate=240]
                    \fill[white] (-0.3,-0.5) rectangle (0.3,0.545);
                \end{scope}

                \draw[darkGreen, dashed, dash pattern={on 2pt off 1pt}, dash phase=2pt] (210:1+\width) .. controls (210:1+\width-0.1) .. (210:\width/2+0.07);
                \draw[darkGreen, dashed, dash pattern={on 2pt off 1pt}, dash phase=2pt] (330:1+\width) .. controls (330:1+\width-0.1) .. (330:\width/2+0.07);

                \begin{scope}
                    \clip (-0.695,0.4) rectangle (0.78,1.44);
                    \draw[darkGreen] plot [smooth cycle, tension=0.4] coordinates {(90:1+\width/2) (210:1+\width/2) (330:1+\width/2)};
                \end{scope}

                \def\edgeA{black}\def\edgeB{black}\def\edgeC{black}\def\edgeD{black}\def\edgeE{black}\def\edgeF{black}
                \def\drawEdgeD{black}\def\drawEdgeE{black}\def\drawEdgeF{black}
                \edges

                \fill[white] (-0.15,-0.6) rectangle (0.15,-1.18);
                \draw[dashed, dash pattern={on 2pt off 1pt}, dash phase=1pt] (-0.15,-0.617) to[out=270+15, in=90-15] (-0.15,-1.162);
                \draw[dashed, dash pattern={on 2pt off 1pt}, dash phase=1pt] ( 0.15,-0.617) to[out=270+15, in=90-15] ( 0.15,-1.162);
                \draw (-0.15,-0.617) to[out=270-20, in=90+20] (-0.15,-1.161);
                \draw ( 0.15,-0.617) to[out=270-20, in=90+20] ( 0.15,-1.161);

                \coordinate (A2) at (270:1+\width/2-0.05);
                \fill[white] ($(A2)-(0.17,0.1)$) rectangle ($(A2)+(0.17,0.1)$);

                \node at ($(A2)-(0.3,0)$) {\small$\alpha_2^+$};
                \node at ($(A2)+(0.3,0)$) {\small$\alpha_2^-$};

                \begin{scope}[xshift=2in, scale=0.7]
                    \def\bigHor{3.5}
                    \def\bigVer{2.25}
                    \draw (0,0) circle [x radius=\bigHor, y radius=\bigVer];
                    \node at (0,\bigVer) [above]{\small$\alpha_4$};

                    \def\ellipseHor{1}
                    \def\ellipseVer{1.5}
                    \def\rad{0.35}
                    \def\horSep{1.6}
                    \def\verSep{0.6}

                    \draw (-\horSep,0) circle [x radius=\ellipseHor, y radius=\ellipseVer];
                    \draw (\horSep,0) circle [x radius=\ellipseHor, y radius=\ellipseVer];

                    \node at (-\horSep,\ellipseVer+0.2) {\small$\alpha_1$};
                    \node at ( \horSep,\ellipseVer+0.2) {\small$\alpha_3$};

                    \def\shift{0.05}

                    \draw[thick, blue] (-\horSep+\rad-0.05,-\verSep-3*\shift) to[out=0, in=180] (\horSep-\rad+0.05,-\verSep+3*\shift);
                    \draw[thick, cyan]
                        (-\horSep+\rad,-\verSep-\shift)
                            .. controls (\ellipseHor,-\verSep-3*\shift) and (\horSep-1,\ellipseVer-0.1) ..
                        (\horSep,\ellipseVer-0.1)
                            .. controls (\horSep,\ellipseVer-0.1) and (\horSep+\ellipseHor-0.1,\ellipseVer-0.1) ..
                        (\horSep+\ellipseHor-0.1,0)
                            .. controls (\horSep+\ellipseHor-0.1,-\ellipseVer+0.1) and (\horSep,-\ellipseVer+0.1) ..
                        (\horSep,-\ellipseVer+0.1)
                            .. controls (\horSep-0.5,-\ellipseVer+0.1) and (\horSep-1,-\verSep-\shift) ..
                        (\horSep-\rad,-\verSep-\shift);
                    \draw[thick, violet]
                        (\horSep-\rad,-\verSep+\shift)
                            .. controls (-\ellipseHor,-\verSep-\shift) and (-\horSep+1,-\ellipseVer+0.1) ..
                        (-\horSep,-\ellipseVer+0.1)
                            .. controls (-\horSep,-\ellipseVer+0.1) and (-\horSep-\ellipseHor+0.1,-\ellipseVer+0.1) ..
                        (-\horSep-\ellipseHor+0.1,0)
                            .. controls (-\horSep-\ellipseHor+0.1,\ellipseVer-0.1) and (-\horSep,\ellipseVer-0.1) ..
                        (-\horSep,\ellipseVer-0.1)
                            .. controls (-\horSep+1,\ellipseVer-0.1) and (-\horSep+1.25,-\verSep+\shift) ..
                        (-\horSep+\rad,-\verSep+\shift);
                    \draw[thick, orange]
                        (-\horSep+\rad-0.2,-\verSep+3*\shift) --
                        (-\horSep+\rad,-\verSep+3*\shift)
                            .. controls (-\horSep+1,-\verSep+3*\shift) and (-\horSep+1,\ellipseVer-0.2) ..
                        (-\horSep,\ellipseVer-0.2)
                            .. controls (-\horSep-\ellipseHor+0.2,\ellipseVer-0.2) and (-\horSep-\ellipseHor+0.2,0.2) ..
                        (-\horSep-\ellipseHor+0.2,0)
                            .. controls (-\horSep-\ellipseHor+0.2,-0.2) and (-\horSep-\ellipseHor+0.2,-\ellipseVer+0.2) ..
                        (-\horSep,-\ellipseVer+0.2)
                            .. controls (-\horSep+0.5,-\ellipseVer+0.2) and (-0.6,-0.87) ..
                        (-0.5,-0.82) to[out=33, in=233] (0.5,-0.03)
                            .. controls (0.7,0.2) and (\horSep-0.5,\ellipseVer-0.2) ..
                        (\horSep,\ellipseVer-0.2)
                            .. controls (\horSep+\ellipseHor-0.2,\ellipseVer-0.2) and (\horSep+\ellipseHor-0.2,0.2) ..
                        (\horSep+\ellipseHor-0.2,0)
                            .. controls (\horSep+\ellipseHor-0.2,-\ellipseVer+0.2) and (\horSep,-\ellipseVer+0.2) ..
                        (\horSep,-\ellipseVer+0.2)
                            .. controls (\horSep-0.4,-\ellipseVer+0.2) and (\horSep-0.8,-\verSep-3*\shift) ..
                        (\horSep-\rad,-\verSep-3*\shift) --
                        (\horSep-\rad+0.2,-\verSep-3*\shift);
                    \draw[thick, darkGreen, dash pattern={on 2pt off 1pt}]
                        (-\horSep+\rad,\verSep)
                            .. controls (-\horSep+0.5,\verSep) and (-\horSep+0.5,\ellipseVer-0.3) ..
                        (-\horSep,\ellipseVer-0.3)
                            .. controls (-\horSep-\ellipseHor+0.3,\ellipseVer-0.3) and (-\horSep-\ellipseHor+0.3,0.3) ..
                        (-\horSep-\ellipseHor+0.3,0)
                            .. controls (-\horSep-\ellipseHor+0.3,-0.3) and (-\horSep-\ellipseHor+0.3,-\ellipseVer+0.3) ..
                        (-\horSep,-\ellipseVer+0.3)
                            .. controls (-\horSep+0.5,-\ellipseVer+0.3) and (-0.6,-0.73) ..
                        (-0.21,-\verSep-0.04)
                            .. controls (1,-0.2) and (\horSep-0.8,\verSep) ..
                        (\horSep-\rad,\verSep);

                    \filldraw[fill=white] (-\horSep, \verSep) circle (\rad);
                    \filldraw[fill=white] (-\horSep,-\verSep) circle (\rad);
                    \filldraw[fill=white] ( \horSep, \verSep) circle (\rad);
                    \filldraw[fill=white] ( \horSep,-\verSep) circle (\rad);

                    \node at ( \horSep, \verSep) {\small$\alpha_6$};
                    \node at ( \horSep,-\verSep) {\small$\alpha_2^-$};
                    \node at (-\horSep, \verSep) {\small$\alpha_5$};
                    \node at (-\horSep,-\verSep) {\small$\alpha_2^+$};
                \end{scope}

                \begin{scope}[xshift=4.25in]
                    \coordinate (A) at (-2,0);
                    \coordinate (B) at (-1,0);
                    \coordinate (C) at (1,0);
                    \coordinate (D) at (2,0);

                    \def\angle{40}

                    \draw[thick, darkGreen, dash pattern={on 2pt off 1pt}] (B) -- (C);
                    \draw[thick, cyan] (A) to[out=40, in=180-\angle] (D);
                    \draw[thick, violet] (A) to[out=-\angle, in=180+\angle] (D);
                    \draw[thick, orange] (A) to[out=-\angle, in=180+\angle] (0,0) to[out=\angle, in=180-\angle] (D);
                    \draw[thick, blue] (A) to[out=\angle, in=180-\angle] (0,0) to[out=-\angle, in=180+\angle] (D);

                    \fill (A) circle (0.02in) node[left] {\small$\alpha_2^+$};
                    \fill (B) circle (0.02in) node[below] {\small$\alpha_5$};
                    \fill (C) circle (0.02in) node[below] {\small$\alpha_6$};
                    \fill (D) circle (0.02in) node[right] {\small$\alpha_2^-$};
                \end{scope}
            \end{tikzpicture}
            \caption{Representing a pair of pants as \protect
                \begin{tikzpicture}
                    \protect\draw (0,0) circle [x radius=0.18, y radius=0.09];
                    \protect\draw (-0.07,0) circle (0.03);
                    \protect\draw (0.07,0) circle (0.03);
                \end{tikzpicture}, which can be nested, the arc $\color{violet}\tpl{1,\slot,0}$ is pictured as one twist around $\alpha_1$, and similarly for $\color{blue}\tpl{0,\slot,0}$, $\color{cyan}\tpl{0,\slot,1}$, and $\color{orange}\tpl{1,\slot,1}$. Collapsing the boundary curves $\alpha_2^\pm$, $\alpha_5$, and $\alpha_6$ induces a system of arcs between four marked points, as shown on the right.}
            \label{fig:exactly_one_type_square_arcs}
        \end{figure}
    \end{lemma}

    \begin{proof}
        Since $\pi\tpl{\epsilon_1,0,\epsilon_3}=\pi\tpl{\epsilon_1,1,\epsilon_3}$ for each $\epsilon_1,\epsilon_3\in\l\{0,1\r\}$, cutting $\Sigma_{1,3}$ along $\alpha_2$ projects $X_C$ onto four distinct arcs, which we denote by $\tpl{\epsilon_1,\slot,\epsilon_3}$ and draw as in Figure \ref{fig:exactly_one_type_square_arcs}. Collapsing the boundary curves $\alpha_2^{\pm}$, $\alpha_5$, and $\alpha_6$ into $4$ marked points induces a system of arcs between those points, from which it is easily seen that there is only one arc joining the vertices $\alpha_5$ and $\alpha_6$ intersecting the other four arcs at most once. Up to a homotopy of the endpoints along the boundaries, we obtain the desired arc $\eta$.\qed
    \end{proof}

    \begin{remark}
        This lemma holds under the action of $G=A_4$, which we leverage below.
    \end{remark}

    \begin{proof}[Proof of Theorem \ref{thm:construction_of_1_system}]
        The collections of curves $X_\triangle$ and $X_\square$ obtained in Propositions \ref{prp:type_triangle_24} and \ref{prp:type_square_3} are individually $1$-systems consisting of $24$ and $3$ curves of types $\triangle$ and $\square$, respectively. Their union with the $6$ disjoint curves $\alpha_1,\dots,\alpha_6$ is then a collection of $33$ curves, which by Proposition \ref{prp:type_triangle_24} is saturated amongst all curves of type $\triangle$. Since any curve $\gamma\not\in\l\{\alpha_1,\dots,\alpha_6\r\}$ in a $1$-system on $\Sigma$ supported by $\mc{P}_0$ induces a cycle on $K_4$, it suffices to show that $X_\triangle\cup X_\square$ is a $1$-system that is saturated amongst all curves of type $\square$ too.

        As in Lemma \ref{lem:exactly_one_type_square}, we cut the subsurface $\Sigma_{1,3}$ corresponding to $C$ along $\alpha_2$ to obtain a subsurface $\Sigma_{0,5}$, from which we obtain a unique arc $\eta$ between $\alpha_5$ and $\alpha_6$ in $\Sigma_{0,5}$ making $\pi(X_C)\cup\l\{\eta\r\}$ a $1$-system of arcs. Similarly, considering the $3$-cycle $C'$ sharing an edge $\alpha_2$ with $C$ gives us a unique arc $\eta'$ between $\alpha_1$ and $\alpha_3$ in the subsurface $\Sigma_{0,5}'$ of $\Sigma$ $-$ obtained similarly as $\Sigma_{0,5}$ $-$ making $\pi(X_{C'})\cup\l\{\eta'\r\}$ a $1$-system of arcs. The two arcs $\eta$ and $\eta'$ uniquely determine the curve $\gamma\coloneqq\tpl{1,\slot,0,\slot,0,1}$ of type $\square$, with $\eta$ fixing the first two twist parameters and $\eta'$ fixing the last two. Glueing the boundary curves $\alpha_2^+$ and $\alpha_2^-$ by an annulus in the obvious way shows that $X_C\cup X_{C'}\cup\l\{\gamma\r\}$ a $1$-system. Moreover, the other two $3$-cycles of $K_4$ both share a common edge $\alpha_4$, and repeating this process gives us the same curve $\gamma$ making $X_\triangle\cup\l\{\gamma\r\}$ a $1$-system.

        By considering the pair $\tpl{\alpha_2,\alpha_4}$ of opposite edges in $K_4$, we have shown that there is exactly one curve $\gamma$ intersecting $\alpha_1$, $\alpha_3$, $\alpha_5$, and $\alpha_6$ making $X_\triangle\cup\l\{\gamma\r\}$ a $1$-system. Repeating this twice by considering the other two pairs of opposite edges $\tpl{\alpha_1,\alpha_6}$ and $\tpl{\alpha_3,\alpha_5}$ furnishes two more curves $\gamma'$ and $\gamma''$, which recovers the collection $X_\square=\l\{\gamma,\gamma',\gamma''\r\}$ of curves of type $\square$ as obtained in Proposition \ref{prp:type_square_3}.\qed
    \end{proof}

    \section{Future work on the Maximality of $X_0$}\label{sec:maximality}

    It would be interesting to establish that the saturated $1$-system $X_0$ as obtained in Theorem \ref{thm:construction_of_1_system} is indeed maximal. We outline a reason for why we believe that it is. Indeed, recall from Section \ref{sec:construction} the correspondence between the trivalent multigraphs on $4$ vertices and the pants decompositions of $\Sigma\coloneqq\Sigma_3$. It can be shown that there are exactly $5$ isomorphism classes of such multigraphs, which are illustrated below.

    \begin{figure}[h]
        \begin{tikzpicture}[scale=0.8]
            \begin{scope}
                \draw (0,-1.73) -- (2,-1.73) -- (1,0) -- (0,-1.73);
                \draw (1,-1.173) -- (0,-1.73);
                \draw (1,-1.173) -- (2,-1.73);
                \draw (1,-1.173) -- (1,0);
                \fill (0,-1.73) circle (0.05cm);
                \fill (2,-1.73) circle (0.05cm);
                \fill (1,0) circle (0.05cm);
                \fill (1,-1.173) circle (0.05cm);
            \end{scope}
            \begin{scope}[xshift=3cm]
                \draw (0,0) -- (1.73,0);
                \draw (0,-1.73) -- (1.73,-1.73);
                \draw (0,0) to[out=250, in=110] (0,-1.73);
                \draw (0,0) to[out=290, in=70] (0,-1.73);
                \draw (1.73,0) to[out=250, in=110] (1.73,-1.73);
                \draw (1.73,0) to[out=290, in=70] (1.73,-1.73);
                \fill (0,0) circle (0.05cm);
                \fill (0,-1.73) circle (0.05cm);
                \fill (1.73,0) circle (0.05cm);
                \fill (1.73,-1.73) circle (0.05cm);
            \end{scope}
            \begin{scope}[xshift=6cm]
                \draw (0,0) to[out=250, in=110] (0,-1.73);
                \draw (0,0) to[out=290, in=70] (0,-1.73);
                \draw (0,0) -- (0.75,-0.865) -- (0,-1.73);
                \draw (0.75,-0.865) -- (1.25,-0.865);
                \draw (1.625,-0.865) circle (0.375cm);
                \fill (0,0) circle (0.05cm);
                \fill (0,-1.73) circle (0.05cm);
                \fill (0.75,-0.865) circle (0.05cm);
                \fill (1.25,-0.865) circle (0.05cm);
            \end{scope}
            \begin{scope}[xshift=9cm, yshift=-0.865cm]
                \draw (0.2,0) circle (0.2cm);
                \draw (0.4,0) -- (0.8,0);
                \draw (0.8,0) to[out=60, in=120] (1.2,0);
                \fill (0.4,0) circle (0.05cm);
                \fill (0.8,0) circle (0.05cm);
                \begin{scope}[xshift=2cm, rotate=180]
                    \draw (0.2,0) circle (0.2cm);
                    \draw (0.4,0) -- (0.8,0);
                    \draw (0.8,0) to[out=60, in=120] (1.2,0);
                    \fill (0.4,0) circle (0.05cm);
                    \fill (0.8,0) circle (0.05cm);
                \end{scope}
            \end{scope}
            \begin{scope}[xshift=13cm, yshift=-0.865cm]
                \fill (0,0) circle (0.05cm);
                \begin{scope}[rotate=90]
                    \draw (0,0) -- (0.346,0);
                    \draw (0.6055,0) circle (0.2595cm);
                    \fill (0.346,0) circle (0.05cm);
                \end{scope}
                \begin{scope}[rotate=210]
                    \draw (0,0) -- (0.346,0);
                    \draw (0.6055,0) circle (0.2595cm);
                    \fill (0.346,0) circle (0.05cm);
                \end{scope}
                \begin{scope}[rotate=330]
                    \draw (0,0) -- (0.346,0);
                    \draw (0.6055,0) circle (0.2595cm);
                    \fill (0.346,0) circle (0.05cm);
                \end{scope}
            \end{scope}
        \end{tikzpicture}
        \caption{The $5$ (isomorphism classes of) trivalent multigraphs on $4$ vertices. The pants decomposition $\mc{P}_0$ of $\Sigma$ whose dual graph is $K_4$ was used to construct $X_0$.}
        \label{fig:5_trivalent_multigraphs}
    \end{figure}

    Each pants decomposition $\mc{P}$ of $\Sigma$ determines $6$ pairwise disjoint curves on $\Sigma$, and conversely each collection of $6$ pairwise disjoint curves on $\Sigma$ determines a pants decomposition of $\Sigma$ in the same way. Thus, for any $1$-system of curves $X$ on $\Sigma$ supported by $\mc{P}$, one might use methods similar to those presented $-$ say by projecting onto cut subsurfaces of $\mc{P}$ $-$ to probe the structure of $X$. Here is a simple result along those lines.

    \begin{proposition}\label{prp:maximal_is_tetrahedral}
        If $X$ is a maximal $1$-system on $\Sigma$ is supported by some pants decomposition $\mc{P}$ of $\Sigma$, then the dual graph of $\mc{P}$ is at least $2$-connected.
    \end{proposition}
    \begin{proof}
        Cut $\Sigma$ along all of its separating pants curves, which decomposes it into subsurfaces of either $\Sigma_{1,1}$, $\Sigma_{1,2}$, and/or $\Sigma_{2,1}$. Any curve $\gamma$ based in one of the separated components must lie within the same component, so this process partitions $X$ into $1$-systems on those components (together with some possibly double-counted pants curves in $\mc{P}$). By \cite{MRT14}*{Theorem 1.2}, we have $N(1,1,n)=3n$, and since $N(1,2)=12$, we have an upper bound $N(1,2,1)\leq N(1,2)+5=17$. Counting the curves shows that if $X$ corresponds to one of the three $1$-connected graphs, then it has at most
        \begin{equation*}
            17+1+3=21,\ \ \ \ 3+1+6+1+3=14,\ \ \ \ \textrm{and}\ \ \ \ 3+3+3+1+1+1=12
        \end{equation*}
        curves, respectively. But $N(1,3)\geq33$ by Theorem \ref{thm:construction_of_1_system}, so none of them can be maximal.
    \end{proof}

    This narrows down the dual graphs of the pants decompositions of $\Sigma$ supporting maximal $1$-systems to the first two graphs in Figure \ref{fig:5_trivalent_multigraphs}. If one proves that the $2$-connected graph can also be eliminated, and that any maximal $1$-system on $\Sigma$ is supported by a pants decomposition of $\Sigma$, then one might be able to study the combinatorics of the curves of types $\triangle$ and $\square$ to show that the $1$-system $X_0$ we constructed is maximal.

    We have attempted to show the last statement, that any $1$-system of curves on $\Sigma$ supported by the pants decomposition whose dual is $K_4$ contains at most $33$ curves. The main obstacle that we encountered seemed to be in balancing the curves of types $\triangle$ and $\square$ differently in Lemma \ref{lem:exactly_one_type_square}; here we have chosen one extreme, to maximize the curves of type $\triangle$. Given the $A_4$-symmetry of $\Sigma$, this seemed to be the most efficient trade-off, but we were unable to formalize this intuition.

\end{document}